\documentclass[11pt,reqno]{amsart}

\usepackage{amsfonts,amssymb,latexsym}
\usepackage{amsmath,amsthm}
\usepackage{eucal}
\usepackage{color}

\usepackage[margin=3cm, a4paper]{geometry}

\newcommand{\EQ}[1]{\begin{equation}\begin{split} #1 \end{split}\end{equation}}
\newcommand{\EQN}[1]{\begin{equation*}\begin{split} #1 \end{split}\end{equation*}}

\def\normo#1{\left\|#1\right\|}

\def\aabs#1{\left|#1\right|}
\def\brk#1{\left(#1\right)}

\def\norm#1{\|#1\|}

\def\wh#1{\widehat{#1}}

\newcommand{\R}{{\mathbb R}}

\newcommand{\Z}{{\mathbb Z}}
\newcommand{\ft}{{\mathcal{F}}}

\newcommand{\les}{{\lesssim}}

\newcommand{\Sch}{{\mathcal{S}}}

\newcommand{\dr}{\omega}

\theoremstyle{plain}
\newtheorem{theorem}[subsection]{Theorem}
\newtheorem{proposition}[subsection]{Proposition}
\newtheorem{lemma}[subsection]{Lemma}

\theoremstyle{remark}

\theoremstyle{definition}

\numberwithin{equation}{section}
\begin{document}

\title[Local well-posedness of the KP-I equation]{Remark on the low regularity well-posedness of the KP-I equation}


\author{Zihua Guo}
\address{School of Mathematics, Monash University, Melbourne, VIC 3800, Australia}
\email{zihua.guo@monash.edu}

\begin{abstract}
We study the Cauchy problem to the KP-I equation posed on $\R^2$.  We prove that it is $C^0$ locally well-posed in $H^{s,0}(\R\times \R)$ for $s>1/2$, which improves the previous results in \cite{GPW,GMo}.  
\end{abstract}

\maketitle

\section{Introduction}

Consider the Cauchy problem for the KP-I equation
\begin{eqnarray}\label{eq:kpI}
\begin{cases}
\partial_tu+\partial_x^3u-\mu \partial_x^{-1}\partial_y^2u+\partial_x(u^2)/2=0;\\
u(x,y,0)=\phi(x,y),
\end{cases}
\end{eqnarray}
where $u(x,y,t):\R^3 \rightarrow \R$ is the unknown function,
$\phi$ is the given initial data and $\mu=1$. The KP-I equation \eqref{eq:kpI}
and the KP-II equation (when $\mu=-1$) arise in
physical contexts as models for the propagation of dispersive long
waves with weak transverse effects. 

It is easy to see that the KP equation is invariant under the following scaling transform: for $\lambda>0$
\EQN{
u(x,y,t)\to \lambda^2 u(\lambda x, \lambda^2 y, \lambda^3 t).
}
So the critical space for KP equation is $\dot H^{s_1,s_2}$ with $s_1+2s_2=-\frac{1}{2}$, in the sense of the scaling invariance.  The KP equation has the following conservation laws: 
\EQN{
M(u) = \int_{\R^2}\! u^{2} dxdy \quad \text{and} \quad 
E(u) = \int_{\R^2}\! \frac{1}{2} u_x^2+\mu \frac{1}{2}  (\partial_x^{-1} u_y)^2-
\frac{1}{6}  u^3 dxdy.
}
In view of the conservation laws, the natural space for the KP-I is $L^2$ or the energy space, while for KP-II the natural space is $L^2$.  Actually, like the KdV equation, the KP equation is completely integrable.  The conserved quantity for KP-II is not coercive apart from $L^2$, while for KP-I it has other coercive quantities.  We will not rely on any integrability structure in this paper. 

The mathematical theory for KP equations was extensively studied. Here we only recall some related results for KP-I equation. It is shown in \cite{MoSaTz1} that the Picard iterative approaches for the KP-I equation fail in standard Sobolev space and $H^{s_1,s_2}$ for $s_1,s_2\in \R$ because the solution flow map is not $C^2$ smooth at the origin, in contrast to the weighted Sobolev space results in \cite{CIKS}. Local well-posedness in $H^{s,0}$ for $s>3/2$ was proved in \cite{MoSaTz4}, and global well-posedness in the ``second" energy spaces was obtained in \cite{Kenig1, MoSaTz2, MoSaTz3}. Ionescu, Kenig and Tataru \cite{IKT}
proved global well-posedness in the natural energy space $\mathbb{E}^1=\{\phi\in
L^2(\R^2),\partial_x \phi \in L^2(\R^2), \partial_x^{-1}\partial_y \phi \in L^2(\R^2)\}$ by introducing the short-time $X^{s,b}$ space and energy estimate method. The result of \cite{IKT} was improved and proof was simplified in \cite{GPW} where local well-posedness was proved in $H^{s,0}$ for $s\geq 1$. In a recent paper \cite{GMo}, the author and Molinet proved unconditional LWP in $H^{s,0}$ for $s>3/4$ and unconditional GWP in the energy space.  In view of the $L^2$ conservation and Bourgain's $L^2$ global well-posedness for KP-II, the $L^2$ well-posedness for KP-I remains a challenging problem. 

The purpose of this note is to study the low regularity well-posedness of KP-I in the anisotropic Sobolev space $H^{s,0}$
which is defined by
\[
H^{s,0}=\{\phi \in L^2(\R^2):
\norm{\phi}_{H^{s,0}}=\norm{\wh{\phi}(\xi,\mu)(1+|\xi|^2)^{s/2}}_{L^2_{\xi,\mu}}<\infty\}.
\]
Note that $L^2=H^{0,0}$.

Now we state our main results:

\begin{theorem}\label{thmmain}

Assume $s>1/2$. Assume $u_0\in H^{\infty,0}(\R^2)$.  Then

(a) There exists
$T=T(\norm{u_0}_{H^{s,0}})>0$ such that there is a unique solution
$u=S^{\infty,0}_T(u_0)\in C([-T,T]:H^{\infty,0})$ of the Cauchy problem
\eqref{eq:kpI}. In addition, for any $\sigma\geq s$
\begin{eqnarray}
\sup_{|t|\leq T} \norm{S^\infty_T(u_0)(t)}_{H^{\sigma,0}}\leq
C(T,\sigma,\norm{u_0}_{H^{\sigma,0}}).
\end{eqnarray}

(b) Moreover, the mapping $S_T^{\infty,0}:H^{\infty,0}\rightarrow
C([-T,T]:H^{\infty,0})$ extends uniquely to a continuous mapping
\[S_T^{s,0}:H^{s,0} \rightarrow C([-T,T]:H^{s,0}).\]
\end{theorem}

The proof of the above theorem is based on some refinements of the arguments in \cite{GPW}, which utilized the method introduced by Ionescu-Kenig-Tataru \cite{IKT}.

\medskip

\noindent {\bf Notations.} We will use the same notations as in \cite{GPW}. 

$\bullet$ For $x, y\in \R^+$, $x\les y$ means that $\exists C>0$ such
that $x\leq Cy$.  

$\bullet$ For $f\in \Sch'$, $\widehat{f}$ or $\ft (f)$ denotes the space-time Fourier
transform of $f$.

$\bullet$ For $k\in \Z$ let
$I_k=\{\xi:|\xi|\in [(3/4)\cdot 2^k,
(3/2)\cdot 2^k)\}$ and $I_{\leq k}=\cup_{j\leq k}I_j$.

$\bullet$ Let $\eta_0: \R\rightarrow [0, 1]$ denote an even smooth function
supported in $[-8/5, 8/5]$ and equal to $1$ in $[-5/4, 5/4]$. For
$k\in \Z$ let $\chi_k(\xi)=\eta_0(|\xi|/2^k)-\eta_0(|\xi|/2^{k-1})$. For $k\in \Z$, $\widehat{P_ku}(\xi)=\chi_{k}(\xi)\widehat{u}(\xi,\mu)$, $P_{\leq k}=\sum_{l\leq k}P_l$

$\bullet$ For $(\xi,\mu)\in \R \setminus \{0\}\times \R$ let
$\dr(\xi,\mu)=\xi^3+\mu^2/\xi$ and 
$\ft_{x,y}[W(t)\phi](\xi,\mu,t)=e^{it\omega(\xi,\mu)}\widehat{\phi}(\xi,\mu)$.

$\bullet$
For $k,j \in \Z_+$ let $D_{k,j}=\{(\xi,\mu, \tau)\in \R^3: \xi \in
{I}_k, \tau-\omega(\xi,\mu)\in {I}_j\},\quad
D_{k,\leq j}=\cup_{l\leq j}D_{k,l}$. Similarly we define $D_{\leq
k,j}$ and $D_{\leq k,\leq j}$. 

$\bullet$
For $k\in \Z_+$, for $f(\xi,\mu,\tau)$ supported in
${I}_k\times\R^2$ (${I}_{\leq 0}\times \R^2$ if $k=0$), we define
\EQN{\norm{f}_{X_{k}}=&\sum_{j=0}^\infty
2^{j/{2}}\norm{\eta_j(\tau-w(\xi,\mu))\cdot
f(\xi,\mu,\tau)}_{L^{2}_{\xi,\mu,\tau}};\\
\norm{f}_{F_k}=&\sup\limits_{t_k\in \R}\norm{\ft[f\cdot
\eta_0(2^{k}(t-t_k))]}_{X_k};\\
\norm{f}_{N_k}=&\sup\limits_{t_k\in
\R}\norm{(\tau-\dr(\xi,\mu)+i2^{k})^{-1}\ft[f\cdot
\eta_0(2^{k}(t-t_k))]}_{X_k}.
}

$\bullet$ For $s\geq 0$ and $T\in (0,1]$, we define the spaces $F^s(T),N^s(T)$ by the norms
\EQN{
\norm{u}_{F^{s}(T)}^2=&\sum_{k=1}^{\infty}2^{2sk}\norm{P_k(u)}_{F_k(T)}^2+\norm{P_{\leq
0}(u)}_{F_0(T)}^2,\\
\norm{u}_{N^{s}(T)}^2=&\sum_{k=1}^{\infty}2^{2sk}\norm{P_k(u)}_{N_k(T)}^2+\norm{P_{\leq
0}(u)}_{N_0(T)}^2.
}

$\bullet$ For $s\in \R$ and $u\in
C([-T,T]:H^{\infty,0})$ we define
\begin{eqnarray*}
\norm{u}_{E^{s}(T)}^2=\norm{P_{\leq 0}(u(0))}_{L^2}^2+\sum_{k\geq
1}\sup_{t_k\in [-T,T]}2^{2sk}\norm{P_k(u(t_k))}_{L^2}^2.
\end{eqnarray*}

$\bullet$ The definition shows easily that if $k\in \Z_+$ and $f_k\in X_k$
then
\begin{eqnarray}\label{eq:pXk1}
\normo{\int_{\R}|f_k(\xi,\mu,\tau')|d\tau'}_{L_\xi^2}\les
\norm{f_k}_{X_k}.
\end{eqnarray}

$\bullet$ We have (see \cite{IKT}, see \cite{GuodgBO} for
a proof) that if $k\in \Z_+$, $l\in \Z_+$, and $f_k\in X_k$ then
\EQ{\label{eq:pXk2}
&\sum_{j=l+1}^\infty
2^{j/2}\normo{\eta_j(\tau-\omega(\xi,\mu))\cdot
\int_{\R}|f_k(\xi,\mu,\tau')|\cdot
2^{-l}(1+2^{-l}|\tau-\tau'|)^{-4}d\tau'}_{L^2}\\
&+2^{l/2}\normo{\eta_{\leq l}(\tau-\omega(\xi,\mu)) \int_{\R}
|f_k(\xi,\mu,\tau')|
2^{-l}(1+2^{-l}|\tau-\tau'|)^{-4}d\tau'}_{L^2}\les \norm{f_k}_{X_k}.
}
In particular, if $k,\ l\in \Z_+$, $t_0\in \R$, $f_k \in X_k$ and
$\gamma\in \Sch(\R)$, then
\begin{eqnarray}\label{eq:pXk3}
\norm{\ft[\gamma(2^l(t-t_0))\cdot \ft^{-1}(f_k)]}_{X_k}\les
\norm{f_k}_{X_k}.
\end{eqnarray}

\section{Bilinear estimates}

In this section, we prove the following bilinear estimates. This improves the bilinear estimates in \cite{GPW}.

\begin{proposition}\label{prop:bilinear}
(a) If $s\geq 0$, $T\in (0,1]$, and $u,v\in F^{s}(T)$ then
\EQ{
\norm{\partial_x(uv)}_{N^{s}(T)}\les&
\norm{u}_{F^{s}(T)}\norm{v}_{F^{0}(T)}+\norm{u}_{F^{0}(T)}\norm{v}_{F^{s}(T)}.
}

(b) If $s\geq 1/2$, $T\in (0,1]$, $u\in F^{s-1}(T)$ and $v\in F^{s}(T)$ then
\EQ{
\norm{\partial_x(uv)}_{N^{s-1}(T)}\les&
\norm{u}_{F^{s-1}(T)}\norm{v}_{F^{s}(T)}.
}
\end{proposition}

The main ingredients are the following dyadic bilinear estimates. 

\begin{lemma}[\cite{GPW}]\label{lem:3zest}
(a) Assume $k_1,k_2,k_3\in \Z$, $j_1,j_2,j_3\in\Z_+$, and
$f_i:\R^3\to\R_+$ are functions supported in $D_{k_i,\leq
j_i}$, $i=1,2,3$.  Assume $f_1,f_2,f_3\in L^{2}(\R^3)$.  
Then
\EQ{\label{eq:3zesta}
\int_{\R^3}(f_1\ast f_2)\cdot f_3\, dxdydt\lesssim
2^{\frac{j_1+j_2+j_3}{2}}\min(2^{-\frac{k_1+k_2+k_3}{2}}, 2^{-\frac{j_{\max}}{2}})
\|f_1\|_{{2}}\|f_2\|_{{2}}\|f_3\|_{2}.
}

(b) Assume $k_1,k_2,k_3\in \Z$ with $k_2\geq 20$, $|k_2-k_3|\leq 5$
and $k_1\leq k_2-10$, $j_1,j_2,j_3\in\Z_+$, and $f_i:\R^3\to\R_+$
are $L^2$ functions supported in $D_{k_i,\leq j_i}$, $i=1,2,3$. Then
if $j_{max}\leq k_1+k_2+k_3-20$,
\begin{eqnarray}\label{eq:3zestb1}
\int_{\R^3}(f_1\ast f_2)\cdot f_3\, dxdydt\lesssim   2^{(j_1+j_2)/2}
2^{-k_3/2}2^{k_1/2}\|f_1\|_{L^{2}}\|f_2\|_{L^{2}}\|f_3\|_{L^2};
\end{eqnarray}
or else if $j_{max}\geq k_1+k_2+k_3-20$,
\begin{eqnarray}\label{eq:3zestb2}
\int_{\R^3}(f_1\ast f_2)\cdot f_3\, dxdydt\lesssim
2^{(j_1+j_2)/2}2^{j_{max}/4}
2^{-k_3}2^{k_1/4}\|f_1\|_{L^2}\|f_2\|_{L^2}\|f_3\|_{L^2}.
\end{eqnarray}
\end{lemma}

\begin{proposition}[Dyadic bilinear estimates]
Assume $u_{k_1}\in F_{k_1}, v_{k_2}\in F_{k_2}$.  Then

(a) high-low to high: If $k_3\geq 20$, $|k_2-k_3|\leq 5$, $0\leq k_1\leq k_2-10$, then
\EQ{\label{eq:hilow}
\norm{P_{k_3}\partial_x(u_{k_1}v_{k_2})}_{N_{k_3}}\les (1+k_1) 2^{-k_1/2}
\norm{u_{k_1}}_{F_{k_1}}\norm{v_{k_2}}_{F_{k_2}}.
}

(b) high-high to high: If $k_3\geq 20$, $|k_3-k_2|\leq 5$, $|k_1-k_2|\leq 5$, then
\EQ{
\norm{P_{k_3}\partial_x(u_{k_1}v_{k_2})}_{N_{k_3}}\les k_3 2^{-k_3/2}
\norm{u_{k_1}}_{F_{k_1}}\norm{v_{k_2}}_{F_{k_2}}.
}

(c) high-high to low:  
If $k_2\geq 20$, $|k_1-k_2|\leq 5$ and $ 0\leq k_3\leq k_1-10$, then
\EQ{
\norm{P_{k_3}\partial_x(u_{k_1}v_{k_2})}_{N_{k_3}}\les k_2 2^{-k_3/2}
\norm{u_{k_1}}_{F_{k_1}}\norm{v_{k_2}}_{F_{k_2}}.
}

(d) low-low to low:  If $0\leq k_1,k_2,k_3\leq 200$, then
\EQ{
\norm{P_{k_3}\partial_x(u_{k_1}v_{k_2})}_{N_{k_3}}\les 
\norm{u_{k_1}}_{F_{k_1}}\norm{v_{k_2}}_{F_{k_2}}.
}   
\end{proposition}

\begin{proof}
(a) By the definitions of $N_k$ and \eqref{eq:pXk2}, we obtain that the left-hand side of \eqref{eq:hilow} is dominated by
\begin{eqnarray}\label{eq:hilowprf1}
&&C\sup_{t_k\in \R}\norm{(\tau-\omega(\xi,\mu)+i2^{k_3})^{-1}\cdot
2^{k_3}1_{I_{k_3}}(\xi) \nonumber\\
&&\quad \cdot \ft[u_{k_1}\eta_0(2^{k_3-2}(t-t_k))]*\ft[v_{k_2}
\eta_0(2^{ k_3-2}(t-t_k))]}_{X_k}.
\end{eqnarray}
To prove part (a), it suffices to prove that if
$j_i\geq  k_3$ and $f_{k_i,j_i}: \R^3\rightarrow \R_+$ are supported
in ${D}_{k_i,\leq j_i}$ ($D_{\leq 0,\leq j_1}$ if $k_1=0$) for
$i=1,2$, then
\EQ{\label{eq:hilowprf2}
&2^{k_3}\sum_{j_3\geq k_3}2^{-j_3/2}\norm{1_{{D}_{k_3,\leq j_3}}\cdot
(f_{k_1,j_1}*f_{k_2,j_2})}_{L^2}\\
&\les  (1+k_1)2^{-{k_1/2}}
2^{(j_1+j_2)/2}\norm{f_{k_1,j_1}}_{L^2}\norm{f_{k_2,j_2}}_{L^2}.
}

Indeed, let $f_{k_1}=\ft[u_{k_1}\eta_0(2^{k_3-2}(t-t_k))]$ and
$f_{k_2}=\ft[v_{k_2}\eta_0(2^{k_3-2}(t-t_k))]$. Then from the
definition of $X_k$ we get that \eqref{eq:hilowprf1} is dominated by
\begin{eqnarray}\label{eq:hilowprf3}
\sup_{t_k\in
\R}2^{k_3}\sum_{j_3=0}^{\infty}2^{j_3/2}\sum_{j_1,j_2\geq k_3}\norm{(2^{j_3}+i2^{k_3})^{-1}1_{{D}_{k_3,j_3}}\cdot
f_{k_1,j_1}*f_{k_2,j_2}}_{L^2},
\end{eqnarray}
where we set
$f_{k_i,j_i}=f_{k_i}(\xi,\mu,\tau)\eta_{j_i}(\tau-\omega(\xi,\mu))$
for $j_i>k_3$ and the remaining part
$f_{k_i,k_3}=f_{k_i}(\xi,\mu,\tau)\eta_{\leq
k_3}(\tau-\omega(\xi,\mu))$, $i=1,2$. For the summation on the terms
$j_3< k_3$ in \eqref{eq:hilowprf3}, we get from the fact
$1_{D_{k_3,j_3}}\leq 1_{{D}_{k_3,\leq j_3}}$ that
\begin{eqnarray}
&&\sup_{t_k\in \R}2^{k_3}\sum_{j_3<k_3}2^{j_3/2}\sum_{j_1,j_2\geq
k_3}\norm{(2^{j_3}+i2^{k_3})^{-1}1_{{D}_{k_3,j_3}}\cdot
f_{k_1,j_1}*f_{k_2,j_2}}_{L^2}\nonumber\\
&&\les \sup_{t_k\in \R}2^{k_3}\sum_{j_1,j_2\geq
k_3}2^{-k_3/2}\norm{1_{{D}_{k_3,\leq k_3}}\cdot
f_{k_1,j_1}*f_{k_2,j_2}}_{L^2}.
\end{eqnarray}
From the fact that $f_{k_i,j_i}$ is supported in ${D}_{k_i,\leq j_i}$
for $i=1,2$ and using \eqref{eq:hilowprf2}, then we get that
\begin{eqnarray*}
&&\sup_{t_k\in \R}2^{k_3}\sum_{j_1,j_2\geq
k_3}2^{-k_3/2}\norm{1_{{D}_{k_3,\leq k_3}}\cdot
f_{k_1,j_1}*f_{k_2,j_2}}_{L^2}\\
 &&\les  (1+k_1)2^{-{k_1/2}}\sup_{t_k\in \R}
\sum_{j_1,j_2\geq
k_3}2^{j_1/2}\norm{f_{k_1,j_1}}_{L^2}2^{j_2/2}\norm{f_{k_2,j_2}}_{L^2}.
\end{eqnarray*}
Thus from the definition and using \eqref{eq:pXk2} and
\eqref{eq:pXk3} we obtain \eqref{eq:hilow}, as desired.

It remains to prove \eqref{eq:hilowprf2}. We assume first $k_1\geq 1$. 
We have
\EQN{
&2^{k_3}\sum_{j_3\geq k_3}2^{-j_3/2}\norm{1_{{D}_{k_3,\leq j_3}}\cdot
(f_{k_1,j_1}*f_{k_2,j_2})}_{L^2}\\
\les& \brk{\sum_{j_3\geq k_3+2k_1}+\sum_{j_3=k_3}^{k_3+2k_1}}2^{k_3}2^{-j_3/2}\norm{1_{{D}_{k_3,\leq j_3}}\cdot
(f_{k_1,j_1}*f_{k_2,j_2})}_{L^2}\\
:=& I+II.
}
For the term $I$, by \eqref{eq:3zestb1} we get
\EQN{
I\les &\sum_{j_3\geq k_3+2k_1}2^{k_3}2^{-j_3/2}\norm{1_{{D}_{k_3,\leq
j_3}}\cdot
(f_{k_1,j_1}*f_{k_2,j_2})}_{L^2}\\
\les& \sum_{j_3\geq k_3+2k_1}2^{k_3}2^{-j_3/2}2^{(j_1+j_2)/2}2^{-k_3/2}2^{k_1/2}\prod_{i=1}^2\norm{f_{k_i,j_i}}_{L^2}\\
\les& 2^{-k_1/2}\prod_{i=1}^2(2^{j_i/2}\norm{f_{k_i,j_i}}_{L^2}).
}
For the term $II$, by \eqref{eq:3zesta} we get
\EQN{
II \les & \sum_{j_3=k_3}^{k_3+2k_1} 2^{k_3}2^{-j_3/2}\norm{1_{{D}_{k_3,\leq
j_3}}\cdot (f_{k_1,j_1}*f_{k_2,j_2})}_{L^2}\\
\les& \sum_{j_3=k_3}^{k_3+2k_1} 2^{k_3}2^{-j_3/2}2^{(j_1+j_2+j_3)/2}2^{-k_3}2^{-k_1/2}\prod_{i=1}^2\norm{f_{k_i,j_i}}_{L^2}\\
\les & k_1 2^{-k_1/2}\prod_{i=1}^2(2^{j_i/2}\norm{f_{k_i,j_i}}_{L^2}).
}

Now we assume $k_1=0$. We apply homogeneous dyadic decomposition to $f_{0, j_1}$ on $\xi$, namely $f_{0,j_1}=\sum_{k_1=-\infty}^0 1_{I_{k_1}}f_{0,j_1}$.  This can be handled in the same way as the term $I$. 
Therefore, we complete the proof for Part (a).

Parts (b) (c) (d) are proved in \cite{GPW}.
\end{proof}

\section{Proof of Theorem \ref{thmmain}}

In this section, we prove Theorem \ref{thmmain} following closely the arguments in \cite{GPW}. First we have the linear estimates.

\begin{proposition}[\cite{GPW}]\label{pFstoHs}
Let $s\geq 0$, $T\in (0,1]$, and $u\in F^{s}(T)$, then
\begin{equation}
\sup_{t\in [-T,T]}\norm{u(t)}_{{H}^{s,0}}\les\ \norm{u}_{F^{s}(T)}.
\end{equation}
\end{proposition}

\begin{proposition}[\cite{GPW}]\label{prop:linear}
Assume $T\in (0,1]$, $u,v\in C([-T,T]:H^{\infty,0})$ and
\begin{eqnarray}\label{eq:lKPI}
u_t+\partial_x^3 u-\partial_x^{-1}\partial_y^2 u=v \mbox{ on } \R^2
\times (-T,T).
\end{eqnarray}
Then for any $s\geq 0$,
\begin{equation}\label{eq:linear}
\norm{u}_{F^{s}(T)}\les \ \norm{u}_{E^{s}(T)}+\norm{v}_{N^{s}(T)}.
\end{equation}
\end{proposition}

The rest of proof of Theorem \ref{thmmain} are the same as in \cite{GPW} if we have the following two propositions

\begin{proposition}\label{prop:energy}
Assume that $T\in (0,1]$ and $u\in C([-T,T]:H^{\infty,0})$ is a
solution to \eqref{eq:kpI} on $\R\times(-T,T)$. Then for $s>1/2$
we have
\begin{eqnarray}\label{eq:energy}
\norm{u}_{E^s(T)}^2\les
\norm{u_0}_{H^{s,0}}^2+\norm{u}_{F^{1/2+}(T)}\norm{u}_{F^{s}(T)}^2.
\end{eqnarray}
\end{proposition}

\begin{proposition}\label{prop:energydiff}
Let $s>1/2$. Let $u_1,u_2 \in F^{s}(1)$ be solutions to \eqref{eq:kpI} with
initial data $\phi_1,\phi_2 \in H^{\infty,0}$ satisfying
\begin{eqnarray}\label{eq:energydiffsmall}
\norm{u_1}_{F^{s}(1)}+\norm{u_2}_{F^{s}(1)}\leq \epsilon_0\ll 1.
\end{eqnarray}
Then we have
\begin{eqnarray}\label{eq:L2conti}
\norm{u_1-u_2}_{F^{s-1}(1)}\les \norm{\phi_1-\phi_2}_{H^{s-1,0}},
\end{eqnarray}
and
\begin{eqnarray}\label{eq:H1conti}
\norm{u_1-u_2}_{F^{s}(1)}\les
\norm{\phi_1-\phi_2}_{H^{s,0}}+\norm{\phi_1}_{H^{s+1,0}}\norm{\phi_1-\phi_2}_{H^{s-1,0}}.
\end{eqnarray}
\end{proposition}

It remains to prove the above propositions. Assume that $u,v\in C([-T,T];L^2)$ and
\begin{equation*}
\begin{cases}
\partial_tu+\partial_x^3u-\partial_x^{-1}\partial_y^2u=v \text{ on }\mathbb{R}^2_{x,y}\times\mathbb{R}_t;\\
u(0)=\phi,
\end{cases}
\end{equation*}
Then we multiply by $u$ and integrate to conclude that
\begin{equation}\label{eq:L2esti}
\sup\limits_{|t_k|\leq T}\norm{u(t_k)}_{L^2}^2\leq
\norm{\phi}_{L^2}^2+\sup\limits_{|t_k|\leq
T}\aabs{\int_{\R\times[0,t_k]}u\cdot v \ dxdydt}.
\end{equation}
In applications we usually take $v=\partial_x(u^2)$. This particular
term has a cancelation that we need to exploit.

\begin{lemma}[\cite{GPW}]\label{lem:3linear}
(a) Assume $T\in (0,1]$, $k_1,k_2,k_3 \in \Z_+$, and $u_i\in
F_{k_i}(T), i=1,2,3$. Then if $k_{min}\leq k_{max}-5$, we have
\begin{eqnarray}\label{eq:3lineara}
\aabs{\int_{\R^2\times [0,T]}u_1u_2u_3\ dxdydt}\les 2^{-k_{min}/2}
\prod_{i=1}^3 \norm{u_i}_{F_{k_i}(T)}.\label{eq:tri1}
\end{eqnarray}

(b) Assume $T\in (0,1]$, $ k\in \Z_+$, $0\leq k_1\leq k-10$, $u \in F^0(T)$ and $v\in F_{k_1}(T)$. Then we have
\begin{eqnarray}\label{eq:3linearb}
\aabs{\int_{\R^2\times
[0,T]}{P}_k(u)\partial_x {P}_k( u \cdot
{P}_{k_1}(v))dxdydt}\les 2^{k_{1}/2}
\norm{v}_{F_{k_1}(T)}\sum_{|k'-k|\leq
10}\norm{{P}_{k'}(u)}_{F_{k'}(T)}^2.
\end{eqnarray}
If $k_1=0$, \eqref{eq:3linearb} also holds if ${P}_{k_1}$
is replaced by ${P}_{\leq 0}$.
\end{lemma}

\begin{proof}[Proof of Proposition \ref{prop:energy}]
From definition we have
\EQN{
\norm{u}_{E^s(T)}^2-\norm{P_{\leq 0}(u_0)}_{L^2}^2\les \sum_{k\geq
1}\sup_{t_k\in [-T,T]}2^{2sk}\norm{{P}_k(u(t_k))}_{L^2}^2.
}
Then we can get from \eqref{eq:L2esti} that
\begin{eqnarray}\label{eq:energyeq}
2^{2sk}\norm{{P}_k(u(t_k))}_{L^2}^2-2^{2sk}\norm{{P}_k(u_0)}_{L^2}^2\les
2^{2sk}\left|\int_{\R\times [0,t_k]}{P}_k(u){P}_k\partial_x(u\cdot
u)dxdt\right|.
\end{eqnarray}
It is easy to see that the right-hand side of \eqref{eq:energyeq} is
dominated by
\EQ{\label{eq:trigoal}
&2^{2sk}\left|\int_{\R\times
[0,t_k]}{P}_k(u)\partial_x{P}_k({P}_{\leq k-10}u\cdot u)dxdt\right|\\
&+2^{2sk}\left|\int_{\R\times
[0,t_k]}{P}_k(u)\partial_x{P}_k({P}_{\geq k-9}u\cdot P_{\leq k-10}u)dxdt\right|\\
&+2^{2sk}\left|\int_{\R\times
[0,t_k]}{P}_k(u)\partial_x{P}_k({P}_{\geq k-9}u\cdot P_{\geq k-9}u)dxdt\right|.
}
For the first two terms in \eqref{eq:trigoal},  using \eqref{eq:3linearb}
then we get that it is bounded by
\EQN{
&2^{2sk}\sum_{k_1\leq k-10} 2^{k_1/2}
\norm{u}_{F_{k_1}(T)}\sum_{|k'-k|\leq
10}\norm{{P}_{k'}(u)}_{F_{k'}(T)}^2\\
\les& \norm{u}_{F^{1/2+}(T)}2^{2sk}\sum_{|k'-k|\leq
10}\norm{u}_{F_{k'}(T)}^2
}
which implies that the summation over $k\geq 1$ is bounded by
$\norm{u}_{F^{1/2+}(T)}\norm{u}_{F^{s}(T)}^2$ as desired.

For the third term in \eqref{eq:trigoal}, using \eqref{eq:tri1} we
get that it is bounded by
\EQN{
& 2^{2sk}\sum_{k_1,k_2\geq k-9}2^k 2^{-k/2}\norm{{P}_k(u)}_{F_k(T)}\norm{{P}_{k_1}(u)}_{F_{k_1}(T)}\norm{{P}_{k_2}(u)}_{F_{k_2}(T)}.
}
Then it is easy to see that summation over $k\geq 1$ is bounded by
$\norm{u}_{F^{1/2+}(T)}\norm{u}_{F^{s}(T)}^2$. We complete the proof of
the proposition.
\end{proof}

\begin{proof}[Proof of Proposition \ref{prop:energydiff}]
We prove first \eqref{eq:L2conti}. Let $v=u_2-u_1$, then $v$ solves
the equation
\begin{eqnarray}\label{eq:energydiffprf1}
\left\{
\begin{array}{l}
\partial_t v+\partial_x^3 v-\partial_x^{-1}\partial_y^2 v=-\partial_x[v(u_1+u_2)/2];\\
v(0)=\phi=\phi_2-\phi_1.
\end{array}
\right.
\end{eqnarray}
Then from Proposition \ref{prop:linear} and Proposition
\ref{prop:bilinear} (b) we obtain
\begin{eqnarray}\label{eq:energydiffprf1}
\left\{
\begin{array}{l}
\norm{v}_{F^{s-1}(1)}\les \norm{v}_{E^{s-1}(1)}+\norm{\partial_x[v(u_1+u_2)/2]}_{N^{s-1}(1)};\\
\norm{\partial_x[v(u_1+u_2)/2]}_{N^{s-1}(1)}\les
\norm{v}_{F^{s-1}(1)}(\norm{u_1}_{F^{s}(1)}+\norm{u_2}_{F^{s}(1)}).
\end{array}
\right.
\end{eqnarray}
We derive an estimate on $\norm{v}_{E^{s-1}(1)}$. As in the proof of
Proposition \ref{prop:energy}, we get from \eqref{eq:L2esti} that
\EQ{\label{eq:energydiffprf2}
&\norm{v}_{E^{s-1}(1)}^2-\norm{\phi}_{E^{s-1}}^2 \\
\les&
\sum_{k\geq1}2^{2k(s-1)}\left|\int_{\R\times
[0,t_k]} {P}_k(v) {P}_k \partial_x (v \cdot
(u_1+u_2))dxdt\right|\\
\les& \sum_{k\geq 1}\sum_{k_1\leq k-10}2^{2k(s-1)}\left|\int_{\R\times
[0,t_k]} {P}_k(v) {P}_k \partial_x [v\cdot {P}_{k_1}(u_1+u_2)] dxdt\right|\\
&+\sum_{k\geq 1}\sum_{k_1\geq k-9}2^{2k(s-1)}\left|\int_{\R\times
[0,t_k]} {P}_k(v) {P}_k \partial_x [v\cdot {P}_{k_1}(u_1+u_2)] dxdt\right|\\
}
For the first term on right-hand side of \eqref{eq:energydiffprf2},
using Lemma \ref{lem:3linear} we can bound it by
\EQN{
&\sum_{k\geq 1}\sum_{k_1\leq k-10}2^{2k(s-1)}\left|\int_{\R\times
[0,t_k]} {P}_k(v) {P}_k \partial_x [v\cdot {P}_{k_1}(u_1+u_2)] dxdt\right|\\
&\les\norm{v}_{F^{s-1}(1)}^2(\norm{u_1}_{F^{1/2+}(1)}+\norm{u_2}_{F^{1/2+}(1)}),
}
The second term on right-hand side of \eqref{eq:energydiffprf2} is
dominated by
\EQN{
&\sum_{k\geq 1}\sum_{k_1\geq k-9}\sum_{k_2\geq 0}2^{2k(s-1)}\left|\int_{\R\times
[0,t_k]} \partial_x {P}_k(v) {P}_k [P_{k_2}v\cdot {P}_{k_1}(u_1+u_2)] dxdt\right|\\
\les&
\norm{v}_{F^{s-1}(1)}^2(\norm{u_1}_{F^{s}(1)}+\norm{u_2}_{F^{s}(1)}).
}
Therefore, we obtain the following estimate
\begin{eqnarray}
\norm{v}_{E^{s-1}(1)}^2\les
\norm{\phi}_{H^{s-1,0}}^2+\norm{v}_{F^{s-1}(1)}^2(\norm{u_1}_{F^{s}(1)}+\norm{u_2}_{F^{s}(1)}),
\end{eqnarray}
which combined with \eqref{eq:energydiffprf2} implies
\eqref{eq:L2conti} in view of \eqref{eq:energydiffsmall}.

We prove now \eqref{eq:H1conti}. From Proposition \ref{prop:linear}
and \ref{prop:bilinear} we obtain
\begin{eqnarray}\label{eq:energydiff2prf1}
\left\{
\begin{array}{l}
\norm{v}_{F^{s}(1)}\les \norm{v}_{E^{s}(1)}+\norm{\partial_x[v(u_1+u_2)/2]}_{N^{s}(1)};\\
\norm{\partial_x[v(u_1+u_2)/2]}_{N^{s}(1)}\les
\norm{v}_{F^{s}(1)}(\norm{u_1}_{F^{s}(1)}+\norm{u_2}_{F^{s}(1)}).
\end{array}
\right.
\end{eqnarray}
Since $\norm{P_{\leq 0}(v)}_{E^{s}(1)}=\norm{P_{\leq
0}(\phi)}_{L^2}$, it follows from \eqref{eq:energydiffsmall} that
\begin{eqnarray}\label{eq:energyv}
\norm{v}_{F^{s}(1)}\les \norm{P_{\geq
1}(v)}_{E^{s}(1)}+\norm{\phi}_{H^{s,0}}.
\end{eqnarray}
To bound $\norm{P_{\geq 1}(v)}_{E^{s}(1)}$, we have
\EQN{
&\norm{v}_{E^{s}(1)}^2-\norm{\phi}_{E^{s}}^2 \\
\les&
\sum_{k\geq1}2^{2ks}\left|\int_{\R\times
[0,t_k]} {P}_k(v) {P}_k \partial_x (v \cdot
(u_1+u_2))dxdt\right|\\
\les& \sum_{k\geq 1}\sum_{k_1\leq k-10}2^{2ks}\left|\int_{\R\times
[0,t_k]} {P}_k(v) {P}_k \partial_x [v\cdot {P}_{k_1}(u_1+u_2)] dxdt\right|\\
&+\sum_{k\geq 1}\sum_{k_1\geq k-9}2^{2ks}\left|\int_{\R\times
[0,t_k]} {P}_k(v) {P}_k \partial_x [v\cdot {P}_{k_1}(u_1+u_2)] dxdt\right|\\
}
For the first term on right-hand side of \eqref{eq:energydiffprf2},
using Lemma \ref{lem:3linear} we can bound it by
\EQN{
&\sum_{k\geq 1}\sum_{k_1\leq k-10}2^{2ks}\left|\int_{\R\times
[0,t_k]} {P}_k(v) {P}_k \partial_x [v\cdot {P}_{k_1}(u_1+u_2)] dxdt\right|\\
&\les\norm{v}_{F^{s}(1)}^2(\norm{u_1}_{F^{1/2+}(1)}+\norm{u_2}_{F^{1/2+}(1)}),
}
The second term on right-hand side of \eqref{eq:energydiffprf2} is
dominated by
\EQN{
&\sum_{k\geq 1}\sum_{k_1\geq k-9}\sum_{k_2\geq 0}2^{2ks}\left|\int_{\R\times
[0,t_k]} \partial_x {P}_k(v) {P}_k [P_{k_2}v\cdot {P}_{k_1}(u_1+u_2)] dxdt\right|\\
\les &\sum_{k\geq 1}\sum_{k_1\geq k-9}\sum_{k_2\geq 0}2^{2ks}\left|\int_{\R\times
[0,t_k]} \partial_x {P}_k(v) {P}_k [P_{k_2}v\cdot {P}_{k_1}(2u_1-v)] dxdt\right|\\
\les&
\norm{v}_{F^{s}(1)}^2(\norm{u_1}_{F^{s}(1)}+\norm{u_2}_{F^{s}(1)})+
\norm{v}_{F^{s-1}(1)}\norm{v}_{F^{s}(1)}\norm{u_1}_{F^{s+1}(1)}.
}
Therefore, we complete the proof. 
\end{proof}

\section*{Acknowledgment}
The author is supported by ARC FT230100588. The author is very grateful to Luc Molinet for the discussions.

\end{document}